\documentclass[preprint,12pt]{elsarticle}

\usepackage{amssymb}
\usepackage{multirow}
\usepackage{amsthm}
\usepackage{graphics}
\usepackage{epsfig}
\usepackage{amssymb}
\usepackage{graphicx}
\usepackage{subfigure}
\usepackage{color}
\usepackage{tikz}
\usetikzlibrary{patterns}
\usepackage{hyperref}
\usepackage{a4wide}
\usepackage{amsmath}
\usepackage{amsfonts}
\usepackage{enumerate}

\newtheorem*{theoremaux}{Theorem \theoremauxnum}
\gdef\theoremauxnum{1}

\newtheorem{lemma}{\bf Lemma}[section]
 
\newtheorem{theorem}{\bf Theorem}[section]
\newtheorem{proposition}[lemma]{\bf Proposition}
\newtheorem{corollary}[lemma]{\bf Corollary}
\newtheorem{definition}{\bf Definition}[section]
\newtheorem{remark}{\bf Remark}[section]

\journal{~}

\begin{document}

\begin{frontmatter}



\title{On Independence Number of Comaximal Subgroup Graph}



\author{Angsuman Das\corref{cor1}}
\ead{angsuman.maths@presiuniv.ac.in}

\author{Arnab Mandal}
\ead{arnab.maths@presiuniv.ac.in}
\address{Department of Mathematics, Presidency University, Kolkata, India}

\cortext[cor1]{Corresponding author}

\begin{abstract}
In this paper, we establish sharp thresholds on the independence number of the comaximal subgroup graph $\Gamma(G)$ that guarantee solvability, supersolvability, and nilpotency of the underlying group $G$. Specifically: \begin{itemize}
    \item For solvability, we prove that any group $G$ with independence number $\alpha(\Gamma(G))\leq 51$ must be solvable, and show that the alternating group $A_5$ is uniquely determined by its graph.
    \item For supersolvability, we show that $\alpha(\Gamma(G))\leq 14$ implies $G$ is supersolvable, except for three explicit exceptions.
    \item For nilpotency, we prove that $\alpha(\Gamma(G))\leq 6$ ensures nilpotency, except for five groups.
\end{itemize}
Finally, we conclude with some open issues involving domination parameters.
\end{abstract}


\end{frontmatter}

\section{Introduction}
The study of graphs associated with algebraic structures provides a powerful framework for investigating the properties of these structures through graph-theoretic methods (See \cite{Cameron-survey} for an extensive survey). A notable example is the comaximal subgroup graph $\Gamma(G)$ of a finite group $G$, introduced in \cite{akbari}. Subsequent research has explored fundamental properties of these graphs, including connectedness, chromatic number, and perfectness (see \cite{1st-paper, 2nd-paper}). Additionally, the comaximal subgroup graph has been examined for specific families of groups, such as dihedral, cyclic groups, as detailed in \cite{dn, zn, deleted-comaximal}.

In this work, we investigate the independence number of the comaximal subgroup graph. Specifically, we establish that imposing upper bounds on the independence number yields structural restrictions on the underlying group, particularly concerning its solvability.

We begin by recalling the definition of the comaximal subgroup graph and some results essential for our analysis.

\begin{definition}\cite{akbari} Let $G$ be a group and $S$ be the collection of all non-trivial proper subgroups of $G$. The co-maximal subgroup graph $\Gamma(G)$ of a group $G$ is defined to be a graph with $S$ as the set of vertices and two distinct vertices $H$ and $K$ are adjacent if and only if $HK=G$. The deleted co-maximal subgroup graph of $G$, denoted by $\Gamma^*(G)$, is defined as the graph obtained by removing the isolated vertices from $\Gamma(G)$.
\end{definition}

\begin{proposition}\label{min-no-subgp-p-groups}
	(Theorem A \& B, \cite{min-no-of-sg}) If $G$ is a non-cyclic group of order $p^k$ with $k\geq 2$, then its number of subgroups satisfy:
	$$Sub(G)\geq \left\lbrace \begin{array}{ll}
	6, & \mbox{ if }k=3 \mbox{ and }p=2 \\
	(k-1)p+(k+1), & \mbox{ otherwise.}
	\end{array} \right.$$ 
\end{proposition}

\begin{proposition}\label{prod-subgp-and-conjugate}
	Let $G$ be a group and $H$ be a proper subgroup of $G$. Then $H\overline{H}\neq G$ for any conjugate $\overline{H}$ of $H$.
\end{proposition}

\begin{proposition}\label{solvable-conjugate-or-permutable}
	(Theorem 14, pg. 452, \cite{ore}) Let $G$ be a finite solvable group and $M,N$ be two maximal subgroups of $G$. Then either $MN=G$ or $M$ and $N$ are conjugates in $G$.
\end{proposition}

\subsection{Our Contribution}
In Section \ref{section-solvable}, we establish that any group whose comaximal subgroup graph has  independence number strictly less than $52$ must be solvable. Furthermore, we demonstrate that the alternating group $A_5$ is uniquely determined by its graph, as it attains the threshold independence number of $52$. Moving to Section \ref{section-supersolvable}, we prove that groups with an independence number below $15$ are supersolvable, with the exception of three groups explicitly listed in Theorem \ref{ind-no-supersolvable-theorem}. In Section \ref{section-nilpotent}, we show that independence number smaller than $7$ guarantees nilpotency, except for five specific groups characterized in Theorem \ref{ind-no-nilpotent-theorem}. In each case, we justify the sharpness of these bounds, ensuring they cannot be further improved. Finally, we conclude by discussing some open problems for future research.

\section{Solvability Criteria for $G$ via Independence Number of $\Gamma(G)$}\label{section-solvable}
In this section, we analyze the relationship between the independence number of a group's comaximal subgroup graph and the solvability of the underlying group. Specifically, we establish precise conditions under which bounds on the independence number guarantee the solvability of the group. Our approach combines graph-theoretic techniques with group-theoretic analysis to derive structural properties of groups from their associated graphs. The following theorem is an improvement (in fact, the best improvement possible) of Theorem 3.1 \cite{2nd-paper} which shows that $\alpha(\Gamma(G))\leq 8$ implies that $G$ is solvable. In fact, the following theorem was left as an open problem in \cite{2nd-paper}.

\begin{theorem}\label{ind-no-solvable-theorem}
Let $G$ be a finite group such that $\alpha(\Gamma(G))\leq 51$. Then $G$ is solvable.
\end{theorem}

\begin{proof}
If possible, let $G$ be a counterexample of minimum order, i.e., $G$ is a non-solvable group of minimum order with $\alpha(\Gamma(G))\leq 51$. As $\alpha(\Gamma(G))\leq 51$, every maximal subgroup $M$ of $G$ contains at most $50$ proper non-trivial subgroups (otherwise $M$ along with its proper non-trivial subgroups forms an independent set of size more than $51$). Thus, $Sub(M)\leq 50+2=52$, i.e., $M$ is solvable (by Theorem 2.1 \cite{das-mandal-number-of-subgroups}). Thus every maximal subgroup of $G$ is solvable and hence, all proper subgroups of $G$ are solvable.

As $G$ is not solvable, then $G$ is a minimal non-solvable group and hence $G/\Phi(G)$ is non-abelian simple group, where $\Phi(G)$ denotes the Frattini subgroup of $G$. 
Now, $\alpha(\Gamma(G/\Phi(G)))\leq \alpha(\Gamma(G))\leq 51$ and $G/\Phi(G)$ is non-solvable. This contradicts the minimality of the counterexample $G$, unless $\Phi(G)$ is trivial. 

Thus $G/\Phi(G)\cong G$ is a non-abelian simple group with all proper subgroups of $G$ are solvable, i.e., $G$ is a finite minimal simple group. Now, by J.G. Thompson's classification of finite minimal simple groups, $G$ is isomorphic to one of the following groups:
\begin{itemize}
    \item $PSL(2,2^p)$, where $p$ is a prime.
    \item $PSL(2,3^p)$, where $p$ is an odd prime.
    \item $PSL(2,p)$, where $p>3$ is prime and $5|p^2+1$.
    \item The Suzuki group $Sz(2^p)$, where $p$ is an odd prime.
    \item $PSL(3,3)$.
\end{itemize}
One can check using SAGE \cite{sage} or otherwise that in each of the cases $\alpha(\Gamma(G))\geq 52$, a contradiction. Hence the theorem holds.
\end{proof}

{\remark Thus for every finite non-solvable group $G$, $\Gamma(G)$ has independence number at least $52$. Moreover, the bound in the above theorem is tight as $\alpha(\Gamma(A_5))=52$. }

In the next theorem, we prove that $A_5$ is uniquely identified from its graph. It is to be noted that this is not true for $A_4$ (See Theorem 5.1 \cite{2nd-paper}), i.e., there exist non-isomorphic groups whose comaximal subgroup graphs are isomorphic to $\Gamma(A_4)$.

\begin{theorem}\label{A5_graph}
Let $G$ be a finite group such that $\Gamma(G)\cong \Gamma(A_5)$. Then $G\cong A_5$. 
\end{theorem}

\begin{proof}
    We start by noting that $\Gamma(A_5)$ is a disjoint union of a complete bipartite graph $K_{5,12}$ and $40$ isolated vertices. Thus $\Gamma(G)$ has $57$ vertices, i.e., $Sub(G)=59$.

    We first show that $G$ is not nilpotent. Suppose $G$ is nilpotent and hence it is the direct product of its Sylow subgroups. Now, as $59$ is prime, $G$ must be a $p$-group. Again, as $\Gamma(G)$ is a bipartite graph, by Theorem 3.7 \cite{1st-paper}, $G$ must be a cyclic $p$-group and hence $\Gamma(G)$ must be edgeless, a contradiction.

    Now, we show that $G$ is not solvable. Suppose $G$ is a non-nilpotent solvable group. So the number of distinct prime factors of $|G|$ is $\geq 2$. If $|G|$ has at least three distinct prime factors $p_1,p_2,p_3$ (say), let $H_{p_1},H_{p_2},H_{p_3}$ be Sylow $p_i$-subgroups of $G$, respectively. Since $G$ is solvable, by Hall's theorem, they have Hall complements $K_{p_1},K_{p_2},K_{p_3}$ (say). Thus we have three edges of the form $H_{p_i}\sim K_{p_i}$ in $\Gamma(G)$. As $\Gamma^*(G)$ is a complete bipartite graph and $H_{p_i}\not\sim H_{p_j}$ for $i\neq j$, we must have $H_{p_1}\sim K_{p_2}$, i.e., $H_{p_1}K_{p_2}=G$. However, $|H_{p_1}K_{p_2}|$ does not contain the factor $p_2$, a contradiction. Thus $|G|$ has exactly two distinct prime factors, i.e., $|G|=p^\alpha q^\beta$. We start by noting that: 
    \begin{itemize}
        \item any Sylow subgroup of $G$ is a vertex of $\Gamma^*(G)$.
        \item Sylow $p$-subgroups and Sylow $q$-subgroups belong to different partite sets of $\Gamma^*(G)$. Let us denote the partite set containing Sylow $p$-subgroups as $V_1$ and the partite set containing Sylow $q$-subgroups as $V_2$.
        \item any $p$-subgroup or $q$-subgroup of $G$ which is not a Sylow subgroup is an isolated vertex in $\Gamma(G)$.
        \item if $A$ is a vertex in $\Gamma^*(G)$, then either $p^\alpha$ or $q^\beta$ divides $|A|$, i.e., $A$ contains a Sylow subgroup of $G$.
    \end{itemize} 

    Since $G$ is solvable of order $p^\alpha q^\beta$, $G$ must have a maximal normal subgroup $M$ of prime index. Let $|M|=p^\alpha q^{\beta-1}$. Thus $M\in V_1$
    
    {\it Claim 1:} The order of any element in $G\setminus M$ is a multiple of $q^\beta$.\\
    {\it Proof of Claim 1:} Let $x \in G\setminus M$ and $H=\langle x \rangle$. Then $MH=G$, i.e., $M\sim H$, i.e., $M\in V_1$ and $H\in V_2$. Thus $|H|$ and hence $\circ(x)$ is a multiple of $q^\beta$. So the Claim holds.
    
    Thus there exists $y_1 \in H\setminus M$ of order $q^\beta$. Hence a Sylow $q$-subgroup $Q_1=\langle y_1 \rangle$ of $G$ is cyclic. Again, as $G$ is a non-nilpotent group, $G$ must have at least three maximal subgroups $M_1,M_2$ and $M_3$, other than $M$ \cite{miller}. Also no two of them are adjacent and hence $M_1,M_2,M_3 \in V_2$. In fact, all maximal subgroups $M_i$ of $G$ other than $M$ belong to $V_2$. Moreover, using Proposition \ref{solvable-conjugate-or-permutable}, any two maximal subgroups in $V_2$ are conjugate and hence are of same order.

    {\it Claim 2:} No Sylow $q$-subgroup of $G$ is maximal in $G$.\\
    {\it Proof of Claim 2:} If possible, Sylow $q$-subgroups of $G$ are maximal in $G$. Then $V_2$ consists only of Sylow $q$-subgroups of $G$. Thus from Claim 1, any element of $G\setminus M$ is of order $q^\beta$. Thus the number of Sylow $q$-subgroups of $G$, $$n_q=\dfrac{|G\setminus M|}{\varphi(q^\beta)}=\dfrac{p^\alpha q^{\beta-1}(q-1)}{q^{\beta-1}(q-1)}=p^\alpha=5, \mbox{ i.e., }\alpha=1, \mbox{ i.e., }N\cong \mathbb{Z}_5.$$

    Again, as $n_q=1+qk=5$, we must have $q=2$. As $G$ has a normal subgroup of order $5$, it has subgroups of order $5,5\cdot 2,5\cdot 2^2,\ldots,5\cdot 2^{\beta-1}$ and all of them belong to $V_1$. Thus $\beta\leq 12$. Hence $G\cong \mathbb{Z}_5\rtimes \mathbb{Z}^\beta_2$ with $\beta\leq 12$.

     Now, let us consider the $40$ isolated vertices in $\Gamma(A_5)$. They correspond to the $2$-subgroups of $G$ which are not Sylow subgroups. Let $H$ and $K$ be two Sylow $2$-subgroup of $G$. Then $5\cdot 2^\beta=|G|>|HK|=\frac{|H||K|}{|H\cap K|}$, i.e., $|H\cap K|\geq 2^{\beta-2}$. Using this fact, it can be shown that $G$ has unique subgroups of order $2,2^2,2^3,\ldots, 2^{\beta-3}$. As $\beta\leq 12$, this add up to at most $9$ isolated vertices in $\Gamma(A_5)$. Thus the remaining at least $31$ isolated vertices correspond to subgroups of order $2^{\beta-2}$ and $2^{\beta-1}$. However we show that we can have at most $5$ subgroups of order $2^{\beta-2}$ and $2^{\beta-1}$ each. Suppose not, then at least two subgroups of order $2^{\beta-2}$ ( or $2^{\beta-1}$) is contained in one of the $5$ Sylow $2$-subgroups of $G$. But this cannot happen in a cyclic group. So we get a contradiction. Hence, Claim 2 holds.
     
    {\it Claim 3:} The Sylow $p$-subgroup is normal in $G$.\\
    {\it Proof of Claim 3:} Let $P$ be a Sylow $p$-subgroup of $M$. Then $P$ is also a Sylow $p$-subgroup of $G$. Then, by using Frattini's argument, we have $G=M\cdot N_G(P)$. Then either $N_G(P)=G$ or $N_G(P)\in V_2$ and $M\sim N_G(P)$. In the later case, as $N_G(P)\in V_2$ we have $q^\beta | |N_G(P)|$. Also as $P\subseteq N_G(P)$, we have $p^\alpha||N_G(P)|$, i.e., $N_G(P)=G$, a contradiction. Thus we have $N_G(P)=G$, i.e., $P\lhd G$. Hence, the claim holds.

    In the rest of the proof, we denote this unique Sylow $p$-subgroup of $G$ by $P$. Also note that, as $G$ is not nilpotent, no Sylow $q$-subgroup of $G$ is normal in $G$.

    {\it Claim 4:} No subgroup in $V_2$ is normal in $G$.\\
    {\it Proof of Claim 4:} Let $H\lhd G$ be in $V_2$ and $Q_H$ be a Sylow $q$-subgroup of $H$. Clearly $Q_H$ is also a Sylow $q$-subgroup of $G$ and $Q_H \in V_2$. Then, by using Frattini's argument, we have $G=H\cdot N_G(Q_H)$. As $Q_H$ is not normal in $G$, we have $N_G(Q_H)\neq G$. Thus $H\sim N_G(Q_H)$ and $N_G(Q_H)\in V_1$. Again as $Q_H\subseteq N_G(Q_H)$, we have $N_G(Q_H)\in V_2$, a contradiction. Hence, the claim holds.

    {\it Claim 5:} $|V_1|=5,|V_2|=12$ and $3\leq \beta\leq 5$.\\
    {\it Proof of Claim 5:} From Claim 2 and Claim 4, it follows that $V_2$ has at least $3$ maximal subgroups and $1+qk$ Sylow $q$-subgroups, thus $|V_2|>5$ and hence $|V_2|=12$ and $|V_1|=5$. Again, since $P$ is a normal subgroup of order $p^\alpha$, $V_1$ contains at least one subgroup each of orders $p^\alpha,p^\alpha q,p^\alpha q^2,\ldots,p^\alpha q^{\beta-1}$ and hence $\beta \leq 5$. Now, if $\beta=1$, then $V_1$ has exactly one element, and if $\beta=2$, then $V_1$ has exactly two elements, namely $P$ and $M$. Both cases lead to a contradiction, and hence $\beta\geq 3$.

    {\it Claim 6:} $p=2$ or $3$.\\
    {\it Proof of Claim 6:} Let $H$ be any subgroup in $V_2$. As $H$ is not normal in $G$ (from Claim 4), $H\subseteq N_G(H)\neq G$. Thus, the number of conjugates of $H$ in $G$ (and in $V_2$)  $=[G:N_G(H)]=p^\alpha q^\beta/|N_G(H)|$ is a multiple of $p$. Thus, $|V_2|=12$ must be equal to the sum of some multiples of $p$. Thus, the only admissible values for $p$ is $2$ or $3$.

    In the rest of the proof, we denote the Sylow $q$-subgroups of $G$ by $Q_i$. Also we note that every $Q_i$ is contained in some maximal subgroup $M_j$ of $G$ where $M_j \in V_2$.

    {\it Claim 7:} The unique maximal subgroup $M \in V_1$ is nilpotent.\\
    {\it Proof of Claim 7:} We prove it separately for the cases $p=2$ and $3$.\\
    {\it Case I: $p=3$.} Since the number of Sylow $q$-subgroups of $G$, $n_q=1+qk$ which divides $3^\alpha$ and $1<n_q\leq |V_2|=12$, $\alpha$ is either $1$ or $2$ and in any case $q=2$. Thus $|G|=3^\alpha\cdot 2^\beta$ and $|M_i|=3^s\cdot 2^\beta$ where $s<\alpha$. Again, since $|V_2|=12$ and the number of conjugates of $M_1$ in $G$ is equal to $[G:N_G(M_1)]=[G:M_1]$, it is either $3$ or $9$. 
    
    If the number of conjugates of $M_1$ is $9$, then $|M_1|=3^{\alpha-2}\cdot 2^\beta$. Also, as $|V_2|=12$, the number of Sylow $q$-subgroups of $G$ is $3$, i.e., $3=[G:N_G(Q_i)]$, i.e., $|N_G(Q_i)|=3^{\alpha-1}\cdot 2 ^\beta$. Thus $V_2$ contains a subgroup $N_G(Q_i)$ of order $3^{\alpha-1}\cdot 2 ^\beta$ which is different from $Q_i$'s and $M_i$'s, a contradiction. Thus the number of conjugates of $M_1$ is $3$ and $|M_1|=3^{\alpha-1}2^\beta$ with $\alpha\geq 2$. Now, we consider the number of Sylow $q$-subgroups of $G$. In this case $n_q=3$ or $9$.

    If $n_q=9$ and if $\alpha\geq 3$, then $|N_G(Q_i)|=3^{\alpha-2}2^\beta$ and $Q_i\subsetneq N_G(Q_i)$. Also we have three $M_i$'s distinct from $N_G(Q_i)$'s in $V_2$. This exceeds the count $|V_2|=12$. On the other hand, when $n_q=9$ and $\alpha=2$, we have $|G|=3^2\cdot 2^\beta$ with $3\leq \beta \leq 5$. However, an exhaustive search on groups of these orders using GAP reveals that none of them yields a graph isomorphic to $\Gamma(A_5)$. Thus $n_q=3$.

    Thus $V_2$ contains three Sylow $2$-subgroups $Q_1,Q_2$ and $Q_3$ of order $2^\beta$ and three maximal subgroups $M_1,M_2$ and $M_3$ of order $3^{\alpha-1}2^\beta$. As $|N_G(Q_i)|=3^{\alpha-1}2^\beta$ and each of $N_G(Q_1)$, $N_G(Q_2)$ and $N_G(Q_3)$ are distinct (since $Q_i$ is a unique Sylow $q$-subgroup of $N_G(Q_i)$), without loss of generality, we can assume $N_G(Q_i)=M_i$ for $i=1,2,3$.

    Now, we show that $G$ has a unique subgroup of order $2^{\beta-1}$. If not, let $L_1,L_2$ be two subgroups of order $2^{\beta-1}$ in $G$. Without loss of generality, let $L_1\subseteq Q_1$ and $L_2\subseteq Q_2$. As $Q_1$ is cyclic, $L_2\not\subset Q_1$. Again as $Q_1$ is the unique Sylow $2$-subgroup of $M_1$, we have $Q_2 \not\subset M_1$. Thus $|M_1|<|M_1Q_2|<|G|$, i.e., $$3^{\alpha-1}2^\beta < \dfrac{3^{\alpha-1}2^\beta \cdot 2^\beta}{|M_1\cap Q_2|}< 3^\alpha 2^\beta.$$ As $M_1\cap Q_2$ is a $2$-group, its order must be $2^{\beta-1}$ to satisfy the above inequality. Thus $L_2$ and $M_1\cap Q_2$ are two subgroups of order $2^{\beta-1}$ in the cyclic group $Q_2$. Hence $L_2=M_1\cap Q_2$ and $L_2\subseteq M_1$. Now, as $Q_1$ is the unique Sylow $2$-subgroup of $M_1$ and $L_2$ is a $2$-subgroup of $M_1$, we must have $L_2\subseteq Q_1$, a contradiction. Thus $G$ has a unique subgroup of order $2^{\beta-1}$, say $L$. So, it follows that $M\cong P\times L$ is nilpotent.

    {\it Case II: $p=2$.} Proceeding as in the above case, the number of Sylow $q$-subgroup, $n_q$ is either $2^2$ or $2^3$. If $n_q=[G:N_G(Q_i)]=8$, then $|N_G(Q_i)|=2^{\alpha-3}q^\beta$. Again, the number of maximal subgroups of $G$ in $V_2$ is $[G:N_G(M_i)]=[G:M_i]=2$ or $4$. However, if $[G:M_i]=2$, then we get a normal subgroup $M_i$ in $V_2$, contradicting Claim 4. Thus there are exactly four maximal subgroups in $V_2$ and all of them are of order $2^{\alpha-2}q^\beta$. If $\alpha>3$, then $N_G(Q_i)\neq Q_i,M_j$ and as a result, $V_2$ contain more than $12$ subgroups, a contradiction. Thus $\alpha=3$. Hence $N_G(Q_i)= Q_i$. Again, as $|M_i|=2q^\beta$, $Q_i$ is normal in $M_i$, i.e., $N_G(Q_i)=M_i$, a contradiction. Hence $n_q=4$ and $q=3$, i.e., $[G:N_G(Q_i)]=4$ and hence $|N_G(Q_i)|=2^{\alpha-2}3^\beta$.

    Again, the number of maximal subgroups of $G$ in $V_2$ is $[G:N_G(M_i)]=[G:M_i]=4$ or $8$. If it is $8$, we have $[G:N_G(M_i)]=[G:M_i]=8$, i.e., $|M_i|=2^{\alpha-3}3^\beta$. As $|N_G(Q_i)|=2^{\alpha-2}3^\beta>|M_i|$, $M_i$ is not a maximal subgroup, a contradiction. Thus there are exactly four maximal subgroups in $V_2$ and all of them are of order $2^{\alpha-2}3^\beta$. As $Q_i$'s are unique normal Sylow $q$-subgroups of $N_G(Q_i)$, we have $Q_i\neq Q_j \Rightarrow N_G(Q_i)\neq N_G(Q_j)$ and without loss of generality $N_G(Q_i)=M_i$ for $i=1,2,3,4$. Now, proceeding similarly as in Case - I, we can show that $G$ has a unique subgroup of order $3^{\beta-1}$ and hence $M$ is nilpotent. Thus Claim 7 holds.

    Thus $M\cong P \times L$, where $L$ is the unique cyclic subgroup of $G$ of order $q^{\beta-1}$. Hence $Sub(M)=Sub(P)\cdot Sub(L)=Sub(P)\cdot \beta$. It is to be noted that all subgroups of $G$, except those in $V_2$ are subgroups of $M$, using Claim 1. Hence $$Sub(M)=1(trivial~subgroup)+40(isolated~vertices)+5(in~V_1)=46=\beta\cdot Sub(P).$$
    Hence the admissible values of $\beta$ are $1,2,23,46$, which contradicts Claim 5.

    Hence $G$ must be a non-solvable group with $Sub(G)=59$. Then by Theorem 2.2 \cite{das-mandal-number-of-subgroups}, we have $G\cong A_5$.
\end{proof}

\section{Supersolvability Criteria for $G$ via Independence Number}\label{section-supersolvable}
In this section, we investigate how the independence number of a group's comaximal subgroup graph relates to the supersolvability of the group. While we establish that groups with sufficiently small independence numbers are typically supersolvable, our results reveal an important distinction from the solvability criteria discussed in Section \ref{section-solvable}: there exist three exceptional cases where this implication fails to hold.

\begin{theorem}\label{ind-no-supersolvable-theorem}
Let $G$ be a finite group such that $\alpha(\Gamma(G))\leq 14$. Then $G$ is supersolvable, or $G$ is isomorphic to one of the following groups: $A_4,SL(2,3),(\mathbb{Z}_2\times \mathbb{Z}_2)\rtimes \mathbb{Z}_9$.
\end{theorem}
\begin{proof}
     Let $G$ be a finite group such that $\alpha(\Gamma(G))\leq 14$. Clearly $G$ is solvable. Suppose $G$ is not supersolvable. Then there exists a maximal subgroup $M$ of $G$ such that $[G:M]=p^m$ where $p$ is a prime with $m>1$. Also $M$ is not normal in $G$, as otherwise $G/M$ will have proper non-trivial normal subgroup contradicting that $M$ is maximal in $G$.
     
     {\it Claim 1:} $[G:M]=4$ or $8$ or $9$.\\
     {\it Proof of Claim 1:} If $p\geq 5$, then we have $[G:M]\geq 25$, i.e., $M$ has at least $25$ conjugates in $G$ and the conjugates forms an independent set in $\Gamma(G)$ (using Proposition \ref{prod-subgp-and-conjugate}) of size $\geq 25$, a contradiction. So $p=2$ or $3$. Using same argument of constructing independent sets using conjugates of $M$ in $G$, we can show that $p=3,m=2$ and $p=2$ with $m=2,3$ are the only possible options.

     {\it Claim 2:} $[G:M]\neq 8$ or $9$.\\
     {\it Proof of Claim 2:} Suppose $[G:M]=8$ or $9$. As $M$ is solvable, $M$ has a normal maximal subgroup, say $X$. As $X$ is normal in $M$ and $M$ is maximal in $G$, we have $N_G(X)=M$ or $G$. If $N_G(X)=M$, then the number of conjugates of $X$ in $G$ is $[G:N_G(X)]=[G:M]\geq 8$. So, $M$ and its conjugates, and $X$ and its conjugates form an independent set of size at least $16$, a contradiction.
     
     So $N_G(X)=G$, i.e., $X$ is normal in $G$. Then $X$ is contained in all conjugates of $M$. If possible, let $Y$ be another maximal subgroup of $M$. If $Y$ is contained in any other conjugate $M_i$ of $M$, then $M=XY=M_i$, a contradiction. Thus $Y$ is contained only in one conjugate of $M$, i.e., $M$ itself. Also other conjugates of $M$ contain a maximal subgroup $Y_i$ which is conjugate to $Y$. By the above argument, the number of such conjugates of $Y$ is at least $8$. Now, $M$ and its conjugates, $X$, $Y$ and its conjugates forms an independent set of size at least $17$ in $\Gamma(G)$, a contradiction. Thus $M$ has a unique maximal subgroup $X$, and hence $M$ is a cylic $p$-group.

     If $[G:M]=8$, then $|G|=2^3\cdot p^n$, where $p>2$. Again, as $M$ is a Sylow $p$-subgroup of $G$, the number of conjugates of $M$ must satisfy $1+pk=8$ which implies $p=7$, i.e., $|G|=2^3\cdot 7^n$. Note that $G$ must have at least $3$ subgroups of order $2$ and at least three subgroups of order $4$, as otherwise we get subgroups of order $2\cdot 7^n$ and $4\cdot 7^n$ containing $M$, contradicting its maximality. Now suppose $n\geq 2$, then $8$ conjugates of $M$, $3$ subgroups of order $2$, $3$ subgroups of order $4$ and one subgroup of order $7$ forms an independent set of size $15$ in $\Gamma(G)$, a contradiction. Thus $n=1$, i.e., $|G|=2^3\cdot 7=56$. Now by checking all non-supersolvable groups of order $56$, we see that none of them admits $\alpha(\Gamma(G))\leq 14$. Thus $[G:M]\neq 8$.

     Similarly as above, if $[G:M]=9$, we get $|G|=3^2\cdot 2^n$. Arguing as above one can show that $n\leq 2$. Thus the possible options for $|G|$ are $3^2\cdot 2$ and $3^2\cdot 2^2$, i.e., $18$ and $36$. Again one can check exhaustively using SAGE \cite{sage} that there does not exist any non-supersolvable group of these orders which admit a cyclic maximal subgroup of index $9$. Thus $[G:M]\neq 9$. Thus Claim 2 holds and we have $[G:M]=4$.

     {\it Claim 3:} $M$ is a cyclic group with at most two distinct prime factors.\\
      {\it Proof of Claim 3:} It is clear that $Sub(M)< 13$, as otherwise non-trivial subgroups of $M$ and the conjugates of $M$ form an independent set of size at least $15$, a contradiction. Thus by \cite{das-mandal-number-of-subgroups}, $M$ is either supersolvable or $M\cong A_4$. In the second case, we get $|G|=48$ and an exhaustive search on non-supersolvable groups of order $48$ shows that for no such group $\alpha(\Gamma(G))\leq 14$ holds. Thus $M$ is supersolvable. 
      
      We first show that $M$ can have at most $3$ maximal subgroups. If possible, let $M$ has at least $4$ maximal subgroups. Let $M=M_1,M_2,M_3,M_4$ be the conjugates of $M$ in $G$. Then each $M_i$ has at least $4$ maximal subgroups. Moreover, any two such maximal subgroups can not belong to two such conjugates of $M$ simultaneously. Let $\mathcal{M}$ be the collection of all maximal subgroups of $M$ and its conjugates. Using simple counting techniques, one can check that $|\mathcal{M}|\geq 10$. Now, the subgroups in $\mathcal{M}$ and conjugates of $M$ form an independent set of size $\geq 14$. This implies $|\mathcal{M}|=10$. Again, if any subgroup in $\mathcal{M}$ has a proper non-trivial subgroup, then that would increase the size of the independent set. So, any subgroup in $\mathcal{M}$ is of prime order. Again, as $M$ is supersolvable, the maximal subgroups of $M$ are of prime index. This means that $|M|=p^2$ or $pq$. Thus $M\cong \mathbb{Z}_{p^2},\mathbb{Z}_{pq},\mathbb{Z}_p\times \mathbb{Z}_p$ or $\mathbb{Z}_p\rtimes \mathbb{Z}_q$. In the first two cases, $M$ has $1$ and $2$ maximal subgroups respectively and hence ruled out. Thus $M\cong \mathbb{Z}_p\times \mathbb{Z}_p$ or $\mathbb{Z}_p\rtimes \mathbb{Z}_q$. Again as $M$ has exactly $4$ maximal subgroups, $M\cong \mathbb{Z}_p\times \mathbb{Z}_p$ implies $p=3$, i.e., $|G|=36$. Now, an exhaustive search on non-supersolvable groups of order $36$ with a maximal subgroup isomorphic to $\mathbb{Z}_3\times \mathbb{Z}_3$ reveals that no such group $G$ exists for which $\alpha(\Gamma(G))\leq 14$. Thus $M\cong \mathbb{Z}_p\rtimes \mathbb{Z}_q$. Since this group has $p+2$ maximal subgroups and $q|p-1$, we must have $p=2$ and $q|1$, a contradiction. Thus $M$ can have at most $3$ maximal subgroups.

      If $M$ has exactly three maximal subgroups then by \cite{khazal} $M$ is either a $2$-group or $M\cong \mathbb{Z}_{p^aq^br^c}$. In the first case, $|G|=|M|[G:M]$ is a power of $2$, i.e, $G$ is a $2$-group, i.e., $G$ is nilpotent, a contradiction. In the second case, as $Sub(M)<13$, we have $M\cong \mathbb{Z}_{pqr}$ or $\mathbb{Z}_{p^2qr}$. In both cases, $M$ has exactly three maximal subgroups, say $X_1,X_2,X_3$, each of them being normal in $M$. Thus $N_G(X_i)=M$ or $G$ for $i=1,2,3$. If at least two of them are normal in $G$, say $X_1,X_2 \lhd G$, then $X_1,X_2$ are contained all the four conjugates of $M$, i.e., all conjugates of $M$ are equal to $X_1X_2$. Thus at most one of $X_1,X_2,X_3$ is normal in $G$. Suppose $X_1,X_2$ are not normal in $G$. Then $N_G(X_1)=N_G(X_2)=M$. So number of conjugates of $X_1$ in $G$ is $[G:N_G(X_1)]=[G:M]=4$. Similarly, $X_2$ also has $4$ conjugates in $G$. These $8$ conjugates of $X_1,X_2$, $4$ conjugates of $M$ and other remaining subgroups of $M$ forms an independent set of size $>14$, a contradiction. Thus $M$ can not have three maximal subgroups.
      
      So $M$ has one or two maximal subgroups and hence  either $M$ is a cyclic $p$-group or a cyclic group with exactly two prime factors. Thus Claim 3 holds.

     {\it Claim 4:} If $M$ is not a $p$-group, then $2\mid |M|$.\\
     {\it Proof of Claim 4:} Note that in this case, $[G:M]=4$ and $M\cong \mathbb{Z}_{p^nq^m}$. Suppose $p,q\neq 2$. Then $|G|=2^2p^nq^m$. Let $M=M_1,M_2,M_3,M_4$ be the conjugates of $M$ in $G$ of order $p^nq^m$. If $G$ has a unique subgroup $H$ of order $2$, then $H\lhd G$ and $MH$ is a proper subgroup of $G$ containing $M$, which contradicts the maximality of $M$ in $G$. Thus $G$ has at least three subgroups $H_1,H_2,H_3$ of order $2$. If both the Sylow $p$ and $q$ subgroups of $G$ are normal, then $M$ is normal in $G$, a contradiction. Again, if both of them are non-normal, we get at least $4+6=10$ Sylow $p$ and $q$-subgroups. This along with $4$ conjugates of $M$, $3$ subgroups of order $2$ creates an independent set of size $\geq 17$, a contradiction. So one of the Sylow $p$ and $q$-subgroups are normal and other is not. Suppose Sylow $p$-subgroups are non-normal and Sylow $q$-subgroup is normal in $G$. Let $Q$ be the unique Sylow $q$-subgroup. Then $QH_i$ for $i=1,2,3$ are subgroups of order $2q^m$ in $G$. Also $QH_i\neq QH_j$, as $QH_i=QH_j$ implies $H_iH_j \subseteq H_i(QH_j)=H_i(QH_i)=QH_i$. On the other hand, the Sylow $2$-subgroup is isomorphic to Klein's $4$-group (because if the Sylow $2$-subgroup is cyclic, and as Sylow $p$ and $q$ subgroups are also cyclic, $G$ becomes supersolvable). Thus $H_iH_j$ is equal to the Sylow $2$-subgroup of $G$ and hence we have $H_iH_j$ is a subgroup of $QH_i$, i.e., $4\mid 2q^m$, a contradiction. Thus $QH_1,QH_2,QH_3$ are three distinct subgroups of order $2q^m$ in $G$. Thus we get an independent set of size at least 15 given by $$4 ~(M_i\mbox{'s})+3 ~(H_i\mbox{'s}) + 3~(QH_i\mbox{'s})+4~(\mbox{atleast }4 \mbox{ Sylow }p\mbox{-subgroups})+1~(Q),$$ a contradiction. Thus $2\mid |M|$ and Claim 4 holds.

     Thus, using Claim 3 and Claim 4, in any case, it follows that $|G|=2^m\cdot p^n$, for an odd prime $p\geq 3$ and $m\geq 2$.

     {\it Claim 5:} $|G|=2^m\cdot 3^n$.\\
     {\it Proof of Claim 5:} Note that the Sylow $p$-subgroup is not normal in $G$, as otherwise we get a subgroup of order $2^{m-1}p^n$ containing $M$, contradicting its maximality in $G$. If $p\geq 5$, then number of Sylow $p$-subgroups is at least $6$. Now, as $M$ is cyclic, there exists at least one Sylow $p$-subgroup, say $H$, which is not contained in any conjugate of $M$. Thus $$|MH|=\dfrac{|M|\cdot |H|}{|M\cap H|}=\dfrac{(2^{m-2}p^n)\cdot p^n}{(\leq p^{n-1})}\geq 2^{m-2}p^{n+1}>2^mp^n=|G|, \mbox{ as }p\geq 5.$$ Thus $p=3$, i.e., Claim 5 holds.

Also from Claim 3, we know that $M$ is a cyclic group with at most two prime factors. If $M$ is a $p$-group, then $M\cong \mathbb{Z}_{3^n}$.

     {\it Claim 6:} If $M\cong \mathbb{Z}_{3^n}$, then $G\cong A_4$ or $(\mathbb{Z}_2\times \mathbb{Z}_2)\rtimes \mathbb{Z}_9$.\\
     {\it Proof of Claim 6:} In this case, we have $|G|=2^2\cdot 3^n$. We first prove that $n\leq 2$. As the Sylow $2$-subgroup of $G$ is isomorphic to $\mathbb{Z}_2\times \mathbb{Z}_2$, we have $3$ subgroups $H_1,H_2,H_3$ of $G$ of order $2$. If $n\geq 3$, then $G$ has subgroups of order $3$ and order $9$ in $G$ which are not Sylow $3$-subgroups of $G$. If both of them are not unique in $G$, then we have at least $4$ subgroups each of order $3$ and order $9$ in $G$. These $8$ subgroups along with $3$ subgroups of order $2$ and $4$ conjugates of $M$ forms an independent set of size $15$, a contradiction. If $G$ has a unique subgroup of order $3$, say $K$, then $H_iK$ are distinct (follows as in proof of Claim 4) subgroups of order $6$ in $G$. Similarly, if $G$ has a unique subgroup of order $3^2$, then we get $3$ distinct subgroups of order $2\cdot 3^2=18$ in $G$. Thus we have an independent set of size either $$3~(\mbox{order }2)+4~(\mbox{conjugates of }M)+1~(\mbox{order }3)+3~(\mbox{order }6)+4~(\mbox{order }3^2)=15, \mbox{ or}$$
     $$3~(\mbox{order }2)+4~(\mbox{conjugates of }M)+1~(\mbox{order }3)+3~(\mbox{order }6)+1~(\mbox{order }3^2)+3~(\mbox{order }18)=15.$$

     As both of them contradict the given bound, we have $n\leq 2$. Thus $|G|=12$ or $36$. Now, running an exhaustive search on non-supersolvable groups of these order in SAGE, one can show that only options for $G$ are $A_4$ or $(\mathbb{Z}_2\times \mathbb{Z}_2)\rtimes \mathbb{Z}_9$.

     {\it Claim 7:} If $M$ is not a $p$-group, then $G\cong SL(2,3)$.\\
     {\it Proof of Claim 7:} As $[G:M]=4$ and $M$ is not a $3$-group, from Claim 5 we have $|M|=2^{m-2}\cdot 3^n$, where $m\geq 3$. Note that Sylow $3$-subgroup is not normal in $G$, as otherwise it becomes a cyclic normal subgroup of a non-supersolvable group, thereby making its quotient (which is a $2$-group) non-supersolvable. Thus $G$ has $4$ Sylow $3$-subgroups. However, as $M$ is cyclic, it contains exactly one subgroup of order $3^n$ and other three Sylow $3$-subgroups of $G$ must lie outside $M$. This enforces $Sub(M)<10$, because $8$ proper non-trivial subgroups of $M$, $3$ Sylow $3$-subgroups of $G$ not in $M$ and $4$ conjugates of $M$ form an independnet set of size $15$. Now, as $M$ is cyclic and $|M|=2^{m-2}\cdot 3^n$, $Sub(M)<10$ implies $(m-1)(n+1)\leq 9$. Thus we have a limited choice for $(m,n)$, namely $(3,1),(3,2),(3,3),(4,1),(4,2)$ and $(5,1)$. The corresponding orders of $G$ are $24,72,216,48,144$ and $96$ respectively. An exhaustive search on non-supersolvable groups of these orders which admit a maximal subgroup $M$ (which is not a $p$-group) of index $4$ using SAGE \cite{sage} yields $G\cong SL(2,3)$.

\end{proof}

We state an immediate corollary, which follows immediately from the above theorem.
\begin{corollary}\label{6-supersolvable}
    Let $G$ be a finite group such that $\alpha(\Gamma(G))\leq 6$. Then $G$ is supersolvable.
\end{corollary}
\begin{proof}
    Since the independence number of the comaximal subgroup graphs corresponding to $A_4,SL(2,3),(\mathbb{Z}_2\times \mathbb{Z}_2)\rtimes \mathbb{Z}_9$ are $7, 12$ and $11$ respectively.
\end{proof}
While we later strengthen this corollary in Theorem \ref{ind-no-nilpotent-theorem} in the next section, this serves as the foundational step for proving Theorem \ref{ind-no-nilpotent-theorem}.

{\remark The upper bound of $14$ is strict in the sense that we have only $3$ (finitely many) exceptions for this bound which are not supersolvable. However we have an infinite family of finite groups, namely $A_4 \times \mathbb{Z}_p$ ($p\geq 5$ is a prime), for which the independence number is $15$. }

{\remark Note that the upper bound provided by Theorem \ref{ind-no-supersolvable-theorem} does not increase when extended to CLT groups (groups satisfying the converse of Lagrange’s theorem). This is because the three exceptions in Theorem \ref{ind-no-supersolvable-theorem} and the infinite family discussed in the preceding remark are all non-CLT groups. }

\section{Nilpotency Criteria for $G$ via Independence Number}\label{section-nilpotent}
In this section, we establish a similar threshold for nilpotency: groups with comaximal subgroup graphs of independence number strictly below $7$ must be nilpotent – except for five special cases. 

\begin{theorem}\label{ind-no-nilpotent-theorem}
Let $G$ be a finite group such that $\alpha(\Gamma(G))\leq 6$. Then either $G$ is nilpotent or $G$ is isomorphic to one of the following groups: $S_3,D_5,\mathbb{Z}_3\rtimes \mathbb{Z}_4,\mathbb{Z}_3\rtimes \mathbb{Z}_8,\mathbb{Z}_5\rtimes \mathbb{Z}_4$.
\end{theorem}
\begin{proof}
    By Corollary \ref{6-supersolvable}, $G$ is supersolvable. \\    
    {\it Claim:} Every maximal subgroup of $G$ is nilpotent. \\
    {\it Proof of Claim:} For that, let $M$ be a maximal subgroup of $G$. Then $M$ has at most $5$ non-trivial proper subgroups (otherwise $M$ along with its non-trivial subgroups will form an independent set of size $\geq 7$). Thus $Sub(M)\leq 7$. If $Sub(M)\neq 6$, by \cite{das-mandal-number-of-subgroups}, $M$ is nilpotent. Suppose $M$ is non-nilpotent and $Sub(M)=6$. Then $M\cong S_3$.

    As $G$ is supersolvable and $M$ is a maximal subgroup of $G$, we have $[G:M]=[G:S_3]=p$ for some prime $p$.

    {\it Case 1:} $p\neq 2,3$. Then $|G|=6p$. Let $P$ be a Sylow $p$-subgroup of $G$. If $n_p>1$, then the Sylow $p$-subgroups and the proper subgroups of $S_3$ forms an independent set in $\Gamma(G)$, i.e., $n_p+4\geq 1+p+4=p+5\geq 10$, a contradiction. Thus $n_p=1$, i.e., $P$ is normal in $G$, i.e., $G\cong \mathbb{Z}_p \rtimes S_3$.    
    If $n_3>1$, then $\mathbb{Z}_p$, at least $4$ Sylow $3$-subgroups and three subgroups of order $2$ in $S_3$ forms an independent set in $\Gamma(G)$ of size $1+4+3=8$, a contradiction. Thus $n_3=1$ and hence $G\cong \mathbb{Z}_{3p}\rtimes \mathbb{Z}_2\cong D_{3p}$. Now $D_{3p}$ has $3p$ subgroups of order $2$, one subgroup of order $3$ and one subgroup of order $p$, i.e., $\alpha(\Gamma(D_{3p}))\geq 3p+2\geq 17$, a contradiction.

    {\it Case 2:} $p=2,3$. Then $|G|=12$ or $18$. However, an exhaustive search on supersolvable groups $G$ of order $12$ and $18$ with $\alpha(\Gamma(G))\leq 6$ reveal that none of them has a maximal subgroup isomorphic to $S_3$.
    
    Hence, the claim holds.

    As every maximal subgroup of $G$ is nilpotent, by Theorem 9.1.9 \cite{robinson-book}, either $G$ is nilpotent or $G$ is minimal non-nilpotent with $|G|=p^mq^n$ and $G$ has a unique Sylow $p$-subgroup $P$ and a cyclic (but not normal) Sylow $q$-subgroup. If $G$ is nilpotent, the theorem holds. In the latter case, we have $n_q\geq 1+q$. Now, if $P$ is not cyclic, then $Sub(P)\geq p+3$. So, the non-trivial proper subgroups of $P$ and Sylow $q$-subgroups of $G$ forms an independent set in $\Gamma(G)$, i.e., $\alpha(\Gamma(G))\geq (p+1)+(q+1)\geq 7$, a contradiction. Thus $P$ must be cyclic. As both Sylow subgroups are cyclic, $G$ is supersolvable and hence a CLT group. Moreover $p>q$.

    Let $Q$ be a Sylow $q$-subgroup with $|Q|=q^n$. If $n\geq 5$, then non-trivial proper subgroups of $Q$ along with conjugates of $Q$ forms an independent set in $\Gamma(G)$, i.e., $\alpha(\Gamma(G))\geq 4+(q+1)\geq 7$, a contradiction. Thus $n\leq 4$. 

    Again, non-trivial proper subgroups of $P$, non-trivial proper subgroups of $Q$ along with conjugates of $Q$ forms an independent set in $\Gamma(G)$, i.e., $6\geq \alpha(\Gamma(G))\geq (m-1)+(n-1)+(q+1)=m+n+q-1$, i.e., $m+n+q\leq 7$ and in particular $q\leq 5$.

    If $q=5$, then $m=n=1$, i.e., $|G|=5p$ and $G \cong \mathbb{Z}_p\rtimes \mathbb{Z}_5$. Also we have $n_5=p$. As Sylow $5$-subgroups forms an independent set, we must have $p\leq 6$. On the other hand, we know that $p>q$. This is a contradiction.

    If $q=3$, then $m+n\leq 4$. Also $1\neq n_q | p^m$, i.e., $n_q\geq p$. As Sylow $q$-subgroups forms an independent set, we must have $p\leq 6$. Now, as $p>q$, we have $p=5$. Thus the only possible option for $G$ is $|G|=5^m\cdot 3^n$ with $m+n\leq 4$. However, an exhaustive search on non-nilpotent groups $G$ of these orders revealed that none of them admits $\alpha(\Gamma(G))\leq 6$.

    If $q=2$, then $m+n\leq 5$. Arguing similarly as above, we get $p=3$ or $5$. Thus the only possible option for $G$ is $|G|=3^m\cdot 2^n$ or $5^m\cdot 2^n$ with $m+n\leq 5$. An exhaustive search on non-nilpotent groups $G$ of these orders revealed that only groups with $\alpha(\Gamma(G))\leq 6$ are $S_3,D_5,\mathbb{Z}_3\rtimes \mathbb{Z}_4,\mathbb{Z}_3\rtimes \mathbb{Z}_8$ and $\mathbb{Z}_5\rtimes \mathbb{Z}_4$. This proves the theorem.
\end{proof}

{\remark The upper bound of $6$ is strict in the sense that we have only $5$ (finitely many) exceptions for this bound which are not nilpotent. However we have an infinite family of finite groups, namely $S_3\times \mathbb{Z}_p$ ($p\geq 3$ is a prime), for which the independence number is $7$. }

\section{Concluding Remarks and Future Directions}
In this work, we established a bridge between combinatorial and algebraic structures by deriving important group-theoretic properties like solvability, supersolvability, and nilpotency from the independence number of the comaximal subgroup graph. Our results demonstrate that graph-theoretic invariants can serve as powerful tools for classifying finite groups, with sharp thresholds and well-characterized exceptions.

While the independence number provides a useful lens for studying groups, other structural properties of the comaximal subgroup graph may also yield further insights:

\begin{enumerate}
    \item Domination Number: Can the domination number detect other group properties? For example, is it true that if the domination number is less than $6$, then the group is supersolvable except for $A_4$ and $SL(2,3)$?
    \item Spectral Aspects: Do the eigenvalues or Laplacian spectrum of the graph reveal hidden group-theoretic features?
\end{enumerate}

\section*{Acknowledgement}
The authors acknowledge the funding of DST-FIST Sanction no. $SR/FST/MS-I/2019/41$ and the first author acknowledges DST-SERB-MATRICS Sanction no. $MTR/$ $2022/000020$, Govt. of India. 

\subsection*{Data Availability Statements}
Data sharing not applicable to this article as no datasets were generated or analysed during the current study.

\subsection*{Competing Interests} The authors have no competing interests to declare that are relevant to the content of this article.

\end{document}